\documentclass[preprint,11pt]{article}
\usepackage{color}
\usepackage{latexsym}
\usepackage{amsmath}
\usepackage{amsthm}
\usepackage{amssymb}
\usepackage{mathrsfs}
\usepackage{authblk} 

\usepackage{subfigure}
\usepackage[graphicx]{realboxes}
\usepackage{caption}

\allowdisplaybreaks[4]

\newtheorem{theorem}{Theorem}[section]
\newtheorem{corollary}[theorem]{Corollary}
\newtheorem{lemma}[theorem]{Lemma}

\newtheorem{remark}[theorem]{Remark}

\begin{document}

\title{\bf{\textmd{Law of the logarithm for the maximum interpoint distance constructed by high-dimensional random matrix}}}

\author{\small Haibin Zhang}
\author{\small Yong Zhang}
\author{\small Xue Ding\thanks{The corresponding address: dingxue83@jlu.edu.cn}}

\affil{School of Mathematics, Jilin University, Changchun 130012, China}

\date{}

\maketitle

\begin{abstract}
Suppose $\left \{ X_{i,k}; 1\le i \le p, 1\le k \le n \right \}  $ is an array of i.i.d.~real random variables. Let  $\left \{ p=p_{n}; n \ge1 \right \}  $ be positive integers. Consider the maximum interpoint distance $M_{n}=\max_{1\le i< j\le p}    \left \| \boldsymbol{X}_{i}- \boldsymbol{X}_{j} \right \|_{2} $ where $\boldsymbol{X}_{i}$ and $\boldsymbol{X}_{j}$ denote the $i$-th and $j$-th rows of the $p \times n$ matrix $\mathcal{M} _{p,n}=\left( X_{i,k} \right)_{p \times n}$, respectively. This paper shows the laws of the logarithm for $M_{n}$ under two high-dimensional settings: the polynomial rate and the exponential rate. The proofs rely on the moderation deviation principle of the partial sum of i.i.d.~random variables, the Chen--Stein Poisson approximation method and Gaussian approximation. \\

\noindent{\bf{Keywords}}
Maximum interpoint distance, law of the logarithm, Chen--Stein Poisson approximation, moderation deviation, Gaussian approximation\\

\noindent{\bf{Mathematics Subject Classification:}} Primary 60F15; secondary 60B12.
\end{abstract}

\section{Introduction \label{sec1}}
Consider a $n$-dimensional population represented by a random vector $\boldsymbol{X}$ with mean $\boldsymbol{\mu}$ and covariance matrix $\boldsymbol{\Sigma}_{n}=\boldsymbol{I}_{n}$, where $\boldsymbol{I}_{n}$ is the $n\times n$ identity matrix. Let $\mathcal{M} _{p,n}=\left(\boldsymbol{X}_{1},\dots, \boldsymbol{X}_{p}  \right)^{\prime}=\left( X_{i,k} \right)_{1\le i \le p, 1\le k \le n}$ be a $p \times n$ random matrix whose  rows are an independent and identically distributed (i.i.d.) random sample of size $p$ from the population $\boldsymbol{X}$. Write $\left \| \cdot \right \|_{2} $ for the Euclidean norm on  $\mathbb{R}^{n}$. Let
\begin{equation}\label{M_{n}}
M_{n}=\max_{1\le i< j\le p} \left \| \boldsymbol{X}_{i}- \boldsymbol{X}_{j} \right \| _{2} 
\end{equation}
denote the maximum interpoint distance or diameter. It is clear that the resulting value of $M_{n}$ will be maximized if we choose two vectors $\boldsymbol{X}_{i}$ and $\boldsymbol{X}_{j}$ with nearly opposing directions and nearly the largest magnitudes, that is, $M_{n}$ is mainly obtained by outliers.  Therefore, this makes $M_{n}$ less useful for goodness of fit tests, but may be appropriate for detecting outliers.

In the past thirty years, several authors mainly studied the limiting distribution of the largest interpoint distance $M_{n}$. For the multidimensional situation, Matthews and Rukhin \cite{MR93} assumed that $\boldsymbol{X}_{1}, \dots,\boldsymbol{X}_{p}$ are independent standard normal random vectors in $\mathbb{R}^{n}$ and obtained the asymptotic distribution of $M_{n}$. Henze and Klein \cite{HK96}  generalized the result of Matthews and Rukhin \cite{MR93} by embedding the multivariate normal distribution into  the symmetric Kotz type distribution and corrected some errors.  Jammalamadaka and Janson \cite{JJ15} considered a more general spherically symmetric distribution and proved that $M_{n}$ with a suitable normalization asymptotically obeys a Gumbel type distribution.

If the distribution of $\boldsymbol{X}$ has an unbounded support, 
Henze and Lao \cite{HL10}  proved that the limiting distribution of $M_{n}$ is none of the three types of classical extreme-value distributions (Gumbel, Weibull and Fr\'echet distributions) when the distribution function $F$ of $\vert \boldsymbol{X} \vert $ satisfies $1-F(s)=s^{-\alpha}L(s)$ as $s \to \infty$ for some $\alpha>0$ and some slowly varying function $L$. Demichel et al.~\cite{DFS15} obtained  the asymptotic behavior of $M_{n}$ when the random vector $\boldsymbol{X}$ in $\mathbb{R}^{n}$ has an elliptical distribution.

In the bounded case, Appel et al.~\cite{ANR02}  discovered if $\boldsymbol{X}$ has a uniform distribution in a planar set with unique major axis and sub-$\sqrt{x}$ decay of its boundary at the endpoints, then the limiting distribution of $M_{n}$ is a convolution of two independent Weibull distributions. Furthermore, Appel and Russo \cite{AR06} investigated the case where i.i.d.~points are uniformly distributed on the surface of a unit hypersphere in $\mathbb{R}^{n}$ and derived the limiting distribution of the maximum pairwise distance.  Mayer and Molchanov \cite{MM07} proved that the limiting distribution of $M_{n}$ is a Weibull distribution when $p$ i.i.d.~points are uniformly distributed in the unit $n$-dimensional ball. Lao \cite{L10} further assumed that $\boldsymbol{X}$ obeys  some distributions in the unit square, the uniform distribution in the unit hypercube, and the uniform distribution in a regular convex polygon. Jammalamadaka and Janson \cite{JJ15} obtained a Gumbel limit distribution for $M_{n}$ if $\boldsymbol{X}_{1},\dots,\boldsymbol{X}_{p} $ are i.i.d.~$\mathbb{R}^{n}$-valued random vectors with a spherically symmetric distribution. Schrempp \cite{S15} relaxed the condition to that random points are in a $n$-dimensional ellipsoid with a unique major axis. Furthermore, Schrempp \cite{S19} considered the case where random points are in a $n$-dimensional set with a unique diameter and a smooth boundary at the poles.  Tang et al.~\cite{TLX22} assumed $\boldsymbol{X}_{1},\dots,\boldsymbol{X}_{p} $ are a random sample from a $n$-dimensional population with independent sub-exponential components and obtained the limiting distribution of $M_{n}$. Heiny and Kleemann \cite{HK23} showed that $M_{n}$ converges weakly to a Gumbel distribution under some moment assumptions and corresponding conditions on the growth rate of $p$.

The aforementioned papers mainly focus on the asymptotic distribution of $M_{n}$ as well as relaxing the range of $p$ relative to $n$.  The logarithmic law for $M_{n}$ remains largely unknown. In this paper, we assume that the sample size $p=p_{n}$   depends on the population dimension  $n$.  We will study the strong limiting theorems for the maximum interpoint distance $M_{n}$ under two high-dimensional cases. They are the polynomial rate with $p_{n}\to \infty$ and $0< c_{1} \le p_{n} /n^{\tau } \le c_{2} <\infty $ and the exponential rate with $p_{n}\to \infty$ and $\log p_{n}=o(n^{\beta})$ where $c_{1}$, $c_{2}$, $\tau$ and $ \beta $ are positive constants. Furthermore, we generalize the result for $M_{n}$ to the maximum interpoint $l^{q}$-norm distance and perform a Monte Carlo simulation to validate our main results.  
 
The laws of the logarithm for  different statistics have previously been studied by some authors, including Jiang \cite{J04} (sample correlation matrix), Li and Rosalsky \cite{LR06} (sample correlation matrix), and Ding \cite{D23} (random tensor). In addition, we refer to Li \cite{L20}, Modarres \cite{M18}, Modarres and Song \cite{MS20}, and Song and Modarres \cite{SM19}  for several other investigates on interpoint distance. 

The plan of this paper is as follows. The main results will be stated in Section \ref{main}. In Section \ref{sec4}, we show some technical lemmas and prove the main results. The $l^{q}$-norm distance and simulation results are given in Section \ref{sec5}.

\section{Main results}\label{main}

Throughout this paper, let $\mathcal{M} _{p,n}=( X_{i,k} )_{1\le i \le p, 1\le k \le n}$ be a $p\times n$ random matrix where $ \{ X_{i,k}; 1\le i \le p, 1\le k \le n \}  $ are i.i.d.~random variables with mean $EX_{1,1}=\mu$ and variance $\operatorname{Var}\left( X_{1,1}\right)=\sigma^{2}>0$, let $\boldsymbol{X}_{1}, \boldsymbol{X}_{2},\dots, \boldsymbol{X}_{p}$ be the $p$ rows of $\mathcal{M} _{p,n}$. And some notations will be used in this paper. The symbol $\overset{P}{\rightarrow} $ implies convergence in probability, $\overset{d}{\rightarrow} $ means convergence in distribution, and $ \to  \text{a.s.}$ means almost sure convergence. For two sequences of positive numbers $\left \{ x_{n} \right \} $ and $\left \{ y_{n} \right \} $, $x_{n}=O\left( y_{n} \right)$ implies $\limsup_{n \to \infty}\vert x_{n}/y_{n}\vert<\infty$, $x_{n}=o\left( y_{n} \right)$ means $\lim_{n\to \infty}x_{n}/y_{n} = 0 $. We write $x_{n} \sim y_{n} $ if $\lim_{n\to \infty} x_{n}/y_{n} = 1 $, and $x_{n} \lesssim y_{n} $ if there exists a universal constant $c>0$ such that $\limsup_{n \to \infty} x_{n}/y_{n} \le c$.
In addition, $C$ is a constant and may vary from line to line.

We first study the law of the logarithm for $M_{n}$ under the assumption that $p$ is a polynomial power of $n$.
\begin{theorem}\label{thm2.3}
Assume
\begin{flalign}
\quad &(i) \  0 < c_{1} \le p/ n^{\tau} \le c_{2} <\infty \ \text{for} \text{ any} \ \tau  >0; \nonumber &\\ 
\quad &(ii) \ E \vert X_{1,1} \vert ^{8\tau+8+\epsilon } < \infty  \ \text{for} \ \text{some} \ \epsilon  >0; \nonumber &\\
\quad &(iii) \ \mathrm{Corr} (\vert X_{1,1}-X_{2,1}\vert^{2}, \vert X_{1,1}-X_{3,1}\vert^{2})<1/3; \nonumber &
\end{flalign}
where $c_{1}$ and $c_{2}$ are constants not depending on $n$. Then, the following holds as $n \to \infty$:
\begin{equation}\label{dingli}
 \frac{  M_{n}^{2}-2 nE\vert X_{1,1}\vert^{2}}{\sqrt{2(E\vert X_{1,1}\vert^{4}+(E\vert X_{1,1}\vert^{2})^{2})n \log_{}{p}}}  \to 2  \quad  \text{a.s.} 
\end{equation}

\end{theorem}

Then, we will consider that $p$ is an exponential power of $n$.
\begin{theorem}\label{thm2.4}
Assume
\begin{flalign}
\quad &(i) \  Ee^{t_{0}\vert X_{1,1} \vert ^{2\alpha}}<\infty \ \text{for some} \ 0<\alpha \le1/2, \ t_{0}>0 ; \nonumber &\\ 
\quad &(ii) \ \log_{}{p}=o\left( n^{\frac{\alpha}{2-\alpha}} \right); \nonumber &\\
\quad &(iii) \ \mathrm{Corr} (\vert X_{1,1}-X_{2,1}\vert^{2}, \vert X_{1,1}-X_{3,1}\vert^{2})<1/3. \nonumber &
\end{flalign}
Then, (\ref{dingli}) holds as $ \to \infty$.
\end{theorem}

If $\boldsymbol{X}$ obeys $N_{n}\left(\boldsymbol{0}, \boldsymbol{I}_{n} \right)$, then $E\vert X_{1,1}\vert^{2}=1$ and $E\vert X_{1,1}\vert^{4}=3$. Therefore, Theorems \ref{thm2.3} and \ref{thm2.4} have the following implication.

\begin{corollary}
Let $\boldsymbol{X}_{1}, \dots, \boldsymbol{X}_{p}$ be a random sample from the standard multivariate normal population. Assume $\log_{}{p}= o(n^{\beta})$ for any $\beta \in (0, 1/3]$ as $n \to \infty$. Then, the following holds as $n \to \infty$:
\begin{equation}
 \frac{  M_{n}^{2}-2 n}{2\sqrt{2 n \log_{}{p}}} \to 2  \quad  \text{a.s.}  \nonumber
\end{equation} 
\end{corollary}

\begin{remark}\label{rem1}
Note that the results of Theorems \ref{thm2.3} and \ref{thm2.4} still hold for arbitrary mean $E\boldsymbol{X}=\mu \boldsymbol{e}$ where $\boldsymbol{e}=\left(1, \dots,1  \right)^{\prime} \in \mathbb{R}^{n} $, since $M_{n}$ is invariant under translation of the vectors $\boldsymbol{X}_{1},\dots,\boldsymbol{X}_{p}$. Therefore, without loss of generality, we assume $ \{ X_{i,k}; 1\le i \le p, 1\le k \le n  \}  $ are i.i.d.~random variables with mean $ E X_{1,1}=\mu=0$ in the proofs. By some calculations, we have 
\begin{equation}
\begin{aligned}
&\mathrm{Corr}  (\vert X_{1,1}-X_{2,1}\vert^{2}, \vert X_{1,1}-X_{3,1}\vert^{2})\\
=&\frac{\mathrm{Cov}  (\vert X_{1,1}-X_{2,1}\vert^{2}, \vert X_{1,1}-X_{3,1}\vert^{2})}{\mathrm{Var} (\vert X_{1,1}-X_{2,1}\vert^{2})}\\
=&\frac{E\vert X_{1,1} \vert^{4}-(E\vert X_{1,1} \vert^{2})^{2}}{2(E\vert X_{1,1} \vert^{4}+(E\vert X_{1,1} \vert^{2})^{2})}. \nonumber
\end{aligned}
\end{equation} 
So the condition $\mathrm{Corr}  (\vert X_{1,1}-X_{2,1}\vert^{2}, \vert X_{1,1}-X_{3,1}\vert^{2})<1/3$ is equivalent to $E\vert X_{1,1} \vert^{4}<5(E\vert X_{1,1} \vert^{2})^{2}$ as shown by Tang et al.~\cite{TLX22}.
\end{remark}

\begin{remark}
In the proof of Theorem \ref{thm2.3}, distinct techniques are employed compared to earlier works like Jiang \cite{J04} and Ding \cite{D23}, which utilized the moderate deviation of the partial sum of i.i.d.~random variables. Instead we use the Gaussian approximation (Za\u{i}tsev \cite{Z87}, Theorem 1.1) in this paper.
\end{remark}

\section{Proofs}\label{sec4}

\subsection{Some technical tools}
We first show the Chen--Stein Poisson approximation method, which is a special case of Theorem 1 from Arratia et al.~\cite{AGG89}.
\begin{lemma}\label{L1}
Let $I$ be an index set and $ \left\{B_{\alpha}; \alpha \in I\right\}   $ be a set of subsets of $  I $, that is, $B_{\alpha} \subset I $ for each $  \alpha \in I$. Let  $ \left\{\eta_{\alpha}; \alpha \in I\right\}  $ be random variables. For a given $  t \in \mathbb{R} $, set  $ \lambda=\sum_{\alpha \in I} P\left(\eta_{\alpha}>t\right)  $. Then we have

\begin{equation}
\left \lvert P\left(\max\limits _{\alpha \in I} \eta_{\alpha} \leq t\right)-e^{-\lambda} \right \rvert \leq\left(1 \wedge \lambda^{-1}\right)\left(b_{1}+b_{2}+b_{3}\right), \nonumber
\end{equation}
where
\begin{equation}
\begin{aligned}
	 b_{1} & = \sum\limits_{\alpha \in I}^{ }  \sum\limits_{\beta \in B_{\alpha }  }^{} P\left ( \eta _{\alpha }> t  \right )P\left ( \eta _{\beta } > t \right )  ,\quad
	 b_{2}  = \sum\limits_{\alpha \in I}^{ }  \sum\limits_{\alpha\neq\beta \in B_{\alpha }  }^{} P\left ( \eta _{\alpha }> t ,  \eta _{\beta } > t \right ) ,\\
	 b_{3} &  =\sum_{\alpha \in I}^{} E \lvert P\left \{ \eta _{\alpha }> t\mid \sigma \left ( \eta _{\beta }; \beta \notin B_{\alpha }   \right )   \right \}- P\left ( \eta _{\alpha }> t  \right ) \rvert ,  \nonumber 
\end{aligned}
\end{equation}
and $\sigma \left ( \eta _{\beta };\beta \notin B_{\alpha }   \right )$ is the $ \sigma $-algebra generated by $ \left \{ \eta _{\beta };\beta \notin  B_{\alpha }   \right \}  $. In particular, if $ \eta _{\alpha } $ is independent of  $ \left \{ \eta _{\beta };\beta \notin  B_{\alpha }   \right \}  $ for each $ \alpha $, then $ b_{3} =0$.
\end{lemma}

The following lemma is about the moderate deviation of the partial sum of i.i.d.~random variables (Linnik \cite{L61}).
\begin{lemma}\label{L2}
Suppose that $ \left\{\zeta, \zeta_{k}; k=1, \dots, n \right\} $ is a sequence of i.i.d.~random variables with $ E \zeta=0 $ and $ E \zeta^{2}=1 $. Define $ S_{n}=\sum_{k=1}^{n} \zeta_{k} $.

(1) If $ E e^{t_{0}|\zeta|^{\alpha}}<\infty $ for some $ 0<\alpha \leq 1 $ and $ t_{0}>0 $. Then
\begin{equation}
\lim _{n \rightarrow \infty} \frac{1}{x_{n}^{2}} \log P\left(\frac{S_{n}}{\sqrt{n}} \geq x_{n}\right)=-\frac{1}{2} \nonumber
\end{equation}
for any $ x_{n} \rightarrow \infty$, $x_{n}=o\left(n^{\frac{\alpha}{2(2-\alpha)}}\right) $.

(2) If $ E e^{t_{0}|\zeta|^{\alpha}}<\infty $ for some $ 0<\alpha \leq 1 / 2 $ and $ t_{0}>0 $. Then
\begin{equation}
\frac{P\left(\frac{S_{n}}{\sqrt{n}} \geq x_{n}\right)}{1-\Phi(x_{n})} \rightarrow 1 \nonumber
\end{equation}
holds uniformly for $ 0 \leq x_{n} \leq o\left(n^{\frac{\alpha}{2(2-\alpha)}}\right) $.
\end{lemma}

\subsection{Proof of Theorem \ref{thm2.3}}

Now we are in a position to prove Theorem \ref{thm2.3}.
\begin{proof}
Review Remark \ref{rem1}. Then, without loss of generality, we prove the theorem by assuming $ \{ X_{i,k}; 1\le i \le p, 1\le k \le n  \}  $ are i.i.d.~random variables with mean $ E X_{1,1}=\mu=0$. Recall (\ref{M_{n}}). Write
\begin{equation}
\begin{aligned}\label{eta}
M_{n}^{2}=\max_{1\le i< j\le p_{n}}  \left \| \boldsymbol{X}_{i}- \boldsymbol{X}_{j} \right \|_{2} ^{2} 
=\max_{1\le i< j\le p_{n}} \sum_{k=1}^{n}\left \vert X_{i,k}-X_{j,k} \right \vert^{2}.
\end{aligned}
\end{equation}
Define 
\begin{equation}
\eta_{ijk}:=\frac{\left \vert X_{i,k}-X_{j,k} \right \vert^{2} -E\vert X_{1,1}-X_{2,1}\vert^{2}}{\sqrt{\mathrm{Var} (\vert X_{1,1}-X_{2,1}\vert^{2})}}=\frac{\left \vert X_{i,k}-X_{j,k} \right \vert^{2} -2E\vert X_{1,1}\vert^{2}}{\sqrt{2 (E\vert X_{1,1}\vert^{4}+(E\vert X_{1,1}\vert^{2})^{2})}}. \nonumber
\end{equation}
It is obvious that 
\begin{equation}\label{etaijk}
\begin{aligned}
E  \eta_{ijk} =0 \quad  \text{and}  \quad
 \text{Var}\left ( \eta_{ijk} \right )=1 \nonumber
\end{aligned}
\end{equation}
for each $k$. Then, we find that 

\begin{equation}\label{etam}
\max_{1\le i< j\le p_{n}} \sum_{k=1}^{n}\eta_{ijk}=\frac{  M_{n}^{2}-2 nE\vert X_{1,1}\vert^{2}}{\sqrt{2(E\vert X_{1,1}\vert^{4}+(E\vert X_{1,1}\vert^{2})^{2})}}. 
\end{equation}
Define 
\begin{equation}\label{hat}
\hat{\eta}_{ijk} :=\eta_{ijk}I_{ \{ \vert \eta_{ijk} \vert \le n^{s}  \} }-E( \eta_{ijk}I_{ \{ \vert \eta_{ijk} \vert \le n^{s}  \} } ), 
\end{equation}
where $0<s<1/2$. To prove Theorem \ref{thm2.3}, it is sufficient to show that
\begin{equation}\label{a.s.one}
\limsup _{n\to\infty}\frac{  \max_{1\le i< j\le p_{n}} \sum_{k=1}^{n}\hat{\eta}_{ijk}}{\sqrt{ n\log_{}{p_{n}}}} \le 2  \quad  \text{a.s.}, 
\end{equation}
\begin{equation}\label{a.s.two}
\liminf _{n\to\infty}\frac{  \max_{1\le i< j\le p_{n}} \sum_{k=1}^{n}\hat{\eta}_{ijk}}{\sqrt{ n \log_{}{p_{n}}}} \ge 2  \quad  \text{a.s.} 
\end{equation}
and 
\begin{equation}\label{a.s.three}
\lim_{n\to\infty}\frac{  \max_{1\le i< j\le p_{n}} \left \vert \sum_{k=1}^{n}\left(\eta_{ijk}- \hat{\eta}_{ijk}\right) \right \vert}{\sqrt{ n \log_{}{p_{n}}}}=0 \quad  \text{a.s.}
\end{equation}
We will complete the proof with three steps.

\emph{Step} 1: \emph{proof of} (\ref{a.s.one}). Given $\varepsilon  \in (0,1)$, set $m_{\tau,\varepsilon}=2/(\tau  \varepsilon)$. For any integer $m >m_{\tau,\varepsilon}  $, we have

\begin{equation}\label{fenjie}
\begin{aligned}
\max_{n^{m}< l\le (n+1)^{m}} \max_{1\le i<j\le p_{l}} \sum_{k=1}^{l} \hat{\eta}_{ijk} 
\le &
\max_{1\le i<j\le p_{(n+1)^{m}}} \max_{n^{m}< l\le (n+1)^{m}}\sum_{k=1}^{l} \hat{\eta}_{ijk}\\
\le &
\max_{1\le i<j \le p_{(n+1)^{m}}} \sum_{k=1}^{n^{m}} \hat{\eta}_{ijk}+\Psi _{n^{m}}, 
\end{aligned}
\end{equation}
where 
\begin{equation}\label{psi}
\begin{aligned}
\Psi _{n^{m}}&=\max_{1\le i<j\le p_{(n+1)^{m}}} \max_{n^{m}< l\le (n+1)^{m}} \left \vert \sum_{k=1}^{l}\hat{ \eta}_{ijk} -\sum_{k=1}^{n^{m}} \hat{\eta}_{ijk} \right \vert \\
&=\max_{1\le i<j\le p_{(n+1)^{m}}} \max_{1 \le h < (n+1)^{m}-n^{m}} \left \vert \sum_{k=1}^{h}   \hat{\eta}_{ijk} \right \vert. 
\end{aligned}
\end{equation}

Set $c_{n1}=\sqrt{\left( 4+\varepsilon \right)  \log_{}{p_{n^{m}}}}$ and $\delta_{n1}=1/(\log_{}{p_{n^{m}}})$. Then, we have
\begin{equation}
\begin{aligned}
P\left( \max_{1\le i<j \le p_{(n+1)^{m}}} \sum_{k=1}^{n^{m}} \hat{\eta}_{ijk} > \sqrt{\left( 4+\varepsilon \right) n^{m} \log_{}{p_{n^{m}}}} \right) 
\le  p_{(n+1)^{m}}^{2} \cdot P\left( \frac{1}{\sqrt{n^{m}}} \sum_{k=1}^{n^{m}}\hat{\eta}_{12k}    > c_{n1} \right). \nonumber
\end{aligned}
\end{equation}
Note that $\vert \hat{\eta}_{ijk} \vert \le n^{s}$.  By Theorem 1.1 from Za\u{i}tsev \cite{Z87}, there exists a random variable $\xi \sim N(0, \text{Var} ( \hat{\eta}_{ijk}  ) )$ such that
\begin{equation}\label{inf}
P\left( \frac{1}{\sqrt{n^{m}}} \sum_{k=1}^{n^{m}}\hat{\eta}_{12k}    > c_{n1} \right)\ge P\left(\xi >c_{n1}+\delta_{n1} \right)- C\text{exp}\left( -C\frac{\sqrt{n^{m}}\delta_{n1}}{n^{sm}} \right),
\end{equation}
\begin{equation}\label{sup}
P\left( \frac{1}{\sqrt{n^{m}}} \sum_{k=1}^{n^{m}}\hat{\eta}_{12k}    > c_{n1} \right)\le P\left(\xi >c_{n1}-\delta_{n1} \right)+ C\text{exp}\left( -C\frac{\sqrt{n^{m}}\delta_{n1}}{n^{sm}} \right).
\end{equation}
For the exponential term, we have
\begin{equation}
p_{(n+1)^{m}}^{2} \cdot \text{exp}\left( -C\frac{\sqrt{n^{m}}\delta_{n1}}{n^{sm}} \right)\le \text{exp}\left( 2m\tau C \log_{}{(n+1)}-\frac{Cn^{(m-2sm)/2}}{\log_{}{p_{n^{m}}}} \right)
\end{equation}
as $n \to \infty$. By the formula that $ \lim_{x\to \infty} P\left ( N\left ( 0,1 \right )\ge x  \right ) =\frac{1}{\sqrt{2\pi }x }e^{-x^{2}/2 } $, we get
\begin{equation}\label{pm}
\begin{aligned}
&~ p_{(n+1)^{m}}^{2} \cdot P\left(\xi >c_{n1}\pm \delta_{n1} \right)\\
\le &~ p_{(n+1)^{m}}^{2} \cdot P\left(\frac{\xi}{\sqrt{\text{Var} ( \hat{\eta}_{ijk}  )}} >c_{n1}\pm \delta_{n1} \right)\\
\sim &~ p_{(n+1)^{m}}^{2} \cdot \left[ 1- \Phi \left(c_{n1}\pm \delta_{n1}\right)  \right]\\
\sim &~\frac{p_{(n+1)^{m}}^{2}}{\sqrt{2\pi}\left( c_{n1}\pm \delta_{n1} \right)}e^{-\frac{\left( c_{n1}\pm \delta_{n1} \right)^{2}}{2}}\\
\sim &~\frac{p_{(n+1)^{m}}^{2}}{\sqrt{2\pi} c_{n1}}e^{-\frac{ c_{n1}^{2}}{2}} = o\left(\frac{1}{n^{m\tau \varepsilon/2}}\right)
\end{aligned}
\end{equation}
as $n \to \infty$. We use the fact that $\text{Var} ( \hat{\eta}_{ijk}  )\le \text{Var} ( \eta_{ijk}  )=1$ in the above inequality. By (\ref{inf})--(\ref{pm}),  we can obtain
\begin{equation}
\sum_{n=1}^{\infty} P\left( \max_{1\le i<j \le p_{(n+1)^{m}}} \sum_{k=1}^{n^{m}} \hat{\eta}_{ijk} > \sqrt{\left( 4+\varepsilon \right) n^{m} \log_{}{p_{n^{m}}}} \right) < \infty \nonumber
\end{equation}
as $n \to \infty$ since $m > m_{\varepsilon} $. By the Borel--Cantelli lemma, we arrive at 
\begin{equation}\label{a.s.1}
\limsup _{n\to\infty}\frac{\max_{1\le i<j \le p_{(n+1)^{m}}} \sum_{k=1}^{n^{m}} \hat{\eta}_{ijk}}{\sqrt{n^{m} \log_{}{p_{n^{m}}}}} \le  \sqrt{4+\varepsilon } \quad  \text{a.s.}
\end{equation}

Trivially, $\frac{ \varepsilon }{2}  \sqrt{ n^{m} \log_{}{p_{n^{m}}}} >  \sqrt{\left( 4+\varepsilon \right)((n+1)^{m}-n^{m}) \log_{}{p_{n^{m}}}}$ due to $0 < c_{1} \le p/ n^{\tau} \le c_{2} <\infty $  for any $ \tau  >0$ as $n$ is sufficiently large. By Markov inequality, we have
\begin{equation}
\begin{aligned}
&\min_{1 \le h < (n+1)^{m}-n^{m}} P\left(  \left \vert  \sum_{k=1}^{(n+1)^{m}-n^{m}} \hat{\eta}_{12k}- \sum_{k=1}^{h} \hat{\eta}_{12k} \right \vert \le \frac{ \varepsilon }{2}   \sqrt{n^{m} \log_{}{p_{n^{m}}}} \right) \\
=& 1-\max_{1 \le h < (n+1)^{m}-n^{m}} P\left(  \left \vert  \sum_{k=1}^{(n+1)^{m}-n^{m}} \hat{\eta}_{12k}- \sum_{k=1}^{h} \hat{\eta}_{12k} \right \vert >\frac{ \varepsilon }{2}  \sqrt{n^{m} \log_{}{p_{n^{m}}}} \right)\\
\ge & 1- \max_{1 \le h < (n+1)^{m}-n^{m}}\frac{C\left( (n+1)^{m}-n^{m}-h \right)}{\varepsilon^{2}n^{m} \log_{}{p_{n^{m}}}}\\
\ge &1- \frac{C\left( (n+1)^{m}-n^{m} \right)}{\varepsilon^{2}n^{m} \log_{}{p_{n^{m}}}} \ge 1/2 \nonumber
\end{aligned}
\end{equation}
as $n\to \infty$.
Then, by (\ref{psi}), Ottaviani's inequality and the same argument as in (\ref{inf})--(\ref{pm}), we obtain
\begin{align}
&~P\left( \Psi _{n^{m}} > \varepsilon  \sqrt{n^{m} \log_{}{p_{n^{m}}}} \right) \nonumber\\
\le  &~ p_{(n+1)^{m}}^{2} \cdot P\left( \max_{1 \le h < (n+1)^{m}-n^{m}} \left \vert \sum_{k=1}^{h} \hat{\eta}_{12k} \right \vert >\varepsilon  \sqrt{n^{m} \log_{}{p_{n^{m}}}} \right) \nonumber\\
\le &~ p_{(n+1)^{m}}^{2} \cdot P\left( \left \vert \sum_{k=1}^{(n+1)^{m}-n^{m}} \hat{\eta}_{12k} \right \vert > \frac{ \varepsilon }{2}  \sqrt{n^{m} \log_{}{p_{n^{m}}}}\right) \nonumber\\
\le &~ p_{(n+1)^{m}}^{2} \cdot P\left( \left \vert \sum_{k=1}^{(n+1)^{m}-n^{m}} \hat{\eta}_{12k} \right \vert > \sqrt{\left( 4+\varepsilon \right)((n+1)^{m}-n^{m}) \log_{}{p_{n^{m}}}} \right) \nonumber\\
= &~ p_{(n+1)^{m}}^{2} \cdot  P\left( \left \vert  \sum_{k=1}^{(n+1)^{m}-n^{m}} \frac{\hat{\eta}_{12k}}{\sqrt{(n+1)^{m} -n^{m}}} \right \vert > \sqrt{\left( 4+\varepsilon \right) \log_{}{p_{n^{m}}}} \right) \nonumber  \\
\lesssim &~C\cdot p_{(n+1)^{m}}^{2} \cdot \left[ 1-\Phi \left(  \sqrt{\left( 4+\varepsilon \right) \log_{}{p_{n^{m}}}} \right) \right]  \nonumber \\
\sim &~ \frac{C\cdot p_{(n+1)^{m}}^{2}}{ \sqrt{2\pi \left( 4+\varepsilon \right) \log_{}{p_{n^{m}}}}}  e^{-\frac{\left( 4+\varepsilon \right)\log_{}{p_{n^{m}}}}{2}}  \nonumber \\
=&~ o\left(\frac{1}{n^{m\tau \varepsilon/2}}\right) \nonumber
\end{align}
for sufficient large $n$. Since $m >m_{\varepsilon} $, we see
\begin{equation}
\sum_{n=1}^{\infty} P\left( \Psi _{n^{m}} > \varepsilon  \sqrt{n^{m} \log_{}{p_{n^{m}}}} \right) < \infty \nonumber
\end{equation}
as $n \rightarrow \infty$. Therefore, by the Borel--Cantelli lemma again,
\begin{equation}\label{a.s.2}
\limsup _{n\to\infty}\frac{\Psi _{n^{m}}}{\sqrt{n^{m} \log_{}{p_{n^{m}}}}} \le \varepsilon  \quad  \text{a.s.}
\end{equation}
Review (\ref{fenjie})--(\ref{psi}) and (\ref{a.s.1})--(\ref{a.s.2}). We obtain that
\begin{equation}
\begin{aligned}
\limsup _{n\to\infty}\frac{ \max_{1\le i<j\le p_{n}} \sum_{k=1}^{n} \hat{\eta}_{ijk}}{\sqrt{n \log_{}{p_{n}}}} & \le \limsup _{n\to\infty}\frac{\max_{n^{m}< l\le (n+1)^{m}} \max_{1\le i<j\le p_{l}} \sum_{k=1}^{l} \hat{\eta}_{ijk}}{\sqrt{n^{m} \log_{}{p_{n^{m}}}}}\\
&\le \sqrt{4+\varepsilon}+\varepsilon  \quad  \text{a.s.} \nonumber
\end{aligned}
\end{equation}
Choosing $\varepsilon >0$ small enough, we then get (\ref{a.s.one}).
 
\emph{Step} 2: \emph{proof of} (\ref{a.s.two}). Given $ \varepsilon \in (0,1)$, set $c_{n2}= \sqrt {\left( 4-\varepsilon \right)  \log_{}{p_{n}}}$.  Set
\begin{equation}
I=\left \{ (i,j);1\le i<j\le p_{n} \right \}. \nonumber
 \end{equation}
For $ \alpha=(i,j)  \in I$, define 
\begin{equation}
X_{\alpha}= \frac{1}{\sqrt{n}} \sum_{k=1}^{n} \hat{ \eta}_{ijk} \nonumber 
\end{equation}
and
\begin{equation}
B_{\alpha }=\left \{ \left ( k,l \right )\in I; \ \left \{ k,l \right \} \cap  \left \{ i,j \right \}\ne  \emptyset , \text{but} \left ( k,l \right )\ne \alpha    \right \}.  \nonumber
\end{equation}
Note that  $ X_{\alpha } $  is independent of $ \left \{ X_{\beta };\beta \notin B_{\alpha } \right \}  $. By Lemma \ref{L1}, we have 
\begin{equation}\label{CSP}
 P\left ( \max_{\alpha \in I}X_{\alpha } \le c_{n2}  \right ) \le e^{-\lambda_{1}  }  +  b_{1}+b_{2},  
\end{equation}
where
\begin{equation}
\begin{aligned}
\lambda_{1}  =\sum_{\alpha \in I}^{}P\left ( X_{\alpha } >  c_{n2} \right ) =
                     \frac{ p_{n}\left ( p_{n}-1 \right )  }{2} P\left ( \frac{1}{\sqrt{n}} \sum_{k=1}^{n} \hat{\eta}_{12k}> c_{n2}   \right ), \nonumber 
\end{aligned}
\end{equation}
\begin{equation}
\begin{aligned}
b_{1} =\sum_{\alpha \in I}^{} \sum_{\beta \in B_{\alpha } }^{} P\left ( X_{\alpha }>  c_{n2}   \right )P\left ( X_{\beta }>  c_{n2}  \right )
         \le  \frac{ p_{n}\left ( p_{n}-1 \right )  }{2} \cdot \left ( 2p_{n} \right )\cdot P\left ( \frac{1}{\sqrt{n}} \sum_{k=1}^{n} \hat{\eta}_{12k} >   c_{n2}   \right )^{2}  
\nonumber
\end{aligned}
\end{equation}
and
\begin{equation}
\begin{aligned}
b_{2}  &=\sum_{\alpha \in I}^{} \sum_{\beta \in B_{\alpha } }^{} P\left ( X_{\alpha }>   c_{n2} , X_{\beta }> c_{n2}   \right)\\ 
	        & \le \frac{ p_{n}\left ( p_{n}-1 \right )  }{2} \cdot \left ( 2p_{n} \right ) \cdot P\left ( \frac{1}{\sqrt{n}} \sum_{k=1}^{n}  \hat{\eta}_{12k} >  c_{n2} ,\frac{1}{\sqrt{n}} \sum_{k=1}^{n}  \hat{\eta}_{13k} > c_{n2}  \right ).  \nonumber
\end{aligned} 
\end{equation}

We first calculate $\lambda_{1}$. By the same argument as in (\ref{inf})--(\ref{pm}), we have
\begin{equation}\label{lambda}
\begin{aligned}
\lambda_{1}=&~\frac{ p_{n}\left ( p_{n}-1 \right )  }{2} P\left ( \frac{1}{\sqrt{n}} \sum_{k=1}^{n} \hat{\eta}_{12k}> c_{n2}   \right )\\
\sim & ~ \frac{p_{n}\left ( p_{n}-1 \right )}{2 \sqrt{2\pi }c_{n2}} e^{-\frac{ c_{n2}^{2}}{2}} \\
=  &~ o \left( n^{\tau \varepsilon /2} \right)
\end{aligned}
\end{equation}
as $n \to \infty$. Then, it follows from (\ref{lambda}) that
\begin{equation}
\begin{aligned}
b_{1} \le &~ p_{n}^{3} \cdot P\left ( \frac{1}{\sqrt{n}} \sum_{k=1}^{n} \hat{\eta}_{12k}> c_{n2}    \right )^{2}\\
\sim&~ \frac{p_{n}^{3}}{2  \pi c_{n2}^{2}} e^{-c_{n2}^{2}} \\
=&~ o\left( \frac{1}{n^{\tau(1-\varepsilon)} } \right)
\end{aligned}
\end{equation}
as $n \to \infty$. 

Finally, we estimate $b_{2}$. Set $\delta_{n2}=1/\sqrt{\log_{}{p_{n}}}$. By Theorem 1.1 from Za\u{i}tsev \cite{Z87}, there exist two normal random variables $\xi_{1}$ and $\xi_{2}$ such that
\begin{equation}
\begin{aligned}
b_{2} \le &~ p_{n}^{3} \cdot P\left ( \min \left \{\frac{1}{\sqrt{n}} \sum_{k=1}^{n}  \hat{\eta}_{12k}, \frac{1}{\sqrt{n}} \sum_{k=1}^{n}  \hat{\eta}_{13k} \right \} > c_{n2}  \right )\\
\le &~ p_{n}^{3} \cdot \left[ P\left ( \min \left \{ \xi_{1}, \xi_{2} \right \} > c_{n2} - \delta_{n2}\right ) +C\text{exp}\left( -C\frac{\sqrt{n}\delta_{n2}}{n^{s}} \right) \right],
\end{aligned}
\end{equation}
where $\{ \xi_{1}, \xi_{2} \} \sim N( 0, \text{Var}(\hat{\eta}_{12k}))$ and $\text{Cov} (\xi_{1}, \xi_{2}) =\text{Cov}(\hat{\eta}_{12k}, \hat{\eta}_{13k}):= \rho_{n}$. For the exponential term, 
\begin{equation}
p_{n}^{3} \cdot \text{exp}\left( -C\frac{\sqrt{n}\delta_{n2}}{n^{s}} \right) \le \text{exp}\left( 3\tau C\log_{}{n}-C n^{\frac{1}{2}-s} (\log_{}{p_{n}})^{-\frac{1}{2}} \right) 
\end{equation}
as $n$ is sufficiently large.
Since $\text{Var}(\xi_{1}) =\text{Var}(\xi_{2})=\text{Var}(\hat{\eta}_{12k})\le \text{Var}(\eta_{12k})=1$, we have
\begin{equation}\label{21}
\begin{aligned}
&~p_{n}^{3} \cdot P\left ( \min \left \{ \xi_{1}, \xi_{2} \right \} > c_{n2} - \delta_{n2}\right ) \\
\le &~ p_{n}^{3} \cdot P\left (  \xi_{1}+ \xi_{2}  >2(c_{n2} - \delta_{n2})  \right )\\
\le &~ p_{n}^{3} \cdot P\left ( \frac{ \xi_{1}+ \xi_{2} }{\sqrt{2\text{Var}(\hat{\eta}_{12k})+2\rho_{n} }} > \frac{2(c_{n2} - \delta_{n2}) }{\sqrt{2+2\rho_{n} }}  \right )\\
\sim &~ \frac{\sqrt{2+2\rho_{n}} \cdot p_{n}^{3} }{2\sqrt{2\pi}c_{n2}} e^{-\frac{c_{n2}^{2}}{1+\rho_{n}}}\\
=&~ o\left( n^{\left(3-\frac{4-\varepsilon}{1+\rho_{n}} \right) \tau} \right)
\end{aligned}
\end{equation}
as $n \to \infty$. By some calculations, we have
\begin{equation}
\begin{aligned}
\rho_{n}&=\text{Cov}\left( \hat{\eta}_{12k}, \hat{\eta}_{13k} \right)\\
&=\text{Cov}\left( \eta_{12k}, \eta_{13k} \right)-E \left(\eta_{12k} \eta_{13k}I _{\{ \max \{\vert \eta_{12k}\vert, \vert \eta_{13k}\vert \}  >n^{s}\} } \right) -\left[ E \left(\eta_{12k} I _{\{  \vert \eta_{12k}\vert >n^{s}\} } \right)\right]^{2}\\
&\to \text{Cov}\left( \eta_{12k}, \eta_{13k} \right) =\text{Corr}  (\vert X_{1,1}-X_{2,1}\vert^{2}, \vert X_{1,1}-X_{3,1}\vert^{2}) \nonumber
\end{aligned}
\end{equation}
as $n \to \infty$. Therefore, if $\text{Corr}  (\vert X_{1,1}-X_{2,1}\vert^{2}, \vert X_{1,1}-X_{3,1}\vert^{2})<1/3$, then $3-\frac{4-\varepsilon}{1+\rho_{n}} <0$ as $n \to \infty$. Next, from (\ref{CSP})--(\ref{21}), we obtain that
\begin{equation}
P\left ( \max_{1\le i <j \le p_{n}}\frac{1}{\sqrt{n}} \sum_{k=1}^{n}  \hat{\eta}_{ijk} \le  c_{n2}  \right )=P\left ( \max_{\alpha \in I}X_{\alpha } \le c_{n2}   \right ) \le o \left( n^{-\varepsilon^{\prime}} \right) \nonumber
\end{equation}
as $n$ is sufficiently large, where $\varepsilon^{\prime}>0$ depends on $\varepsilon$ and $\text{Corr}  (\vert X_{1,1}-X_{2,1}\vert^{2}, \vert X_{1,1}-X_{3,1}\vert^{2})$. Take an integer $m>1/\varepsilon^{\prime}$ such that
\begin{equation}
\sum_{n=1}^{\infty} P\left( \max_{1\le i<j \le p_{n^{m}}}\frac{1}{\sqrt{n^{m}}} \sum_{k=1}^{n^{m}} \hat{\eta}_{ijk} \le  \sqrt{\left( 4-\varepsilon \right) \log_{}{p_{n^{m}}}} \right) \le \sum_{n=1}^{\infty} o \left( n^{-m\varepsilon^{\prime}} \right) < \infty \nonumber
\end{equation}
as $n \to \infty$. Then, by the Borel--Cantelli lemma, 
\begin{equation}\label{a.s.3}
\liminf _{n\to\infty}\frac{ \max_{1\le i<j\le p_{n^{m}}} \sum_{k=1}^{n^{m}} \eta_{ijk}}{\sqrt{n^{m} \log_{}{p_{n^{m}}}}} > \sqrt{4-\varepsilon } \quad  \text{a.s.} 
\end{equation}
Recalling (\ref{psi}), we find
\begin{equation}
\inf_{n^{m}< l\le (n+1)^{m}} \max_{1\le i<j\le p_{l}} \sum_{k=1}^{l} \hat{\eta}_{ijk} \ge \max_{1\le i<j \le p_{n^{m}}} \sum_{k=1}^{n^{m}} \hat{\eta}_{ijk} - \Psi _{n^{m}}. \nonumber
\end{equation}
Combining the above inequality, (\ref{a.s.2}) and (\ref{a.s.3}), we obtain that
\begin{equation}
\begin{aligned}
&\liminf _{n\to\infty} \frac{  \max_{1\le i<j\le p_{n}} \sum_{k=1}^{n} \eta_{ijk}}{\sqrt{n \log_{}{p_{n}}}}\\
\ge & \liminf _{n\to\infty} \frac{ \inf_{n^{m}< l\le (n+1)^{m}} \max_{1\le i<j\le p_{l}} \sum_{k=1}^{l} \eta_{ijk}}{\sqrt{n^{m} \log_{}{p_{n^{m}}}}}\\
\ge & \sqrt{4-\varepsilon }-\varepsilon \quad  \text{a.s.} \nonumber
\end{aligned}
\end{equation}
Then, (\ref{a.s.two}) follows from the arbitrariness of $\varepsilon$.

\emph{Step} 3: \emph{proof of} (\ref{a.s.three}). Reviewing (\ref{hat}), we have
\begin{equation}
E\left \vert \eta_{12k}-\hat{\eta}_{12k} \right \vert ^{4\tau +4 +\epsilon } \lesssim E\left [ \vert \eta_{12k} \vert ^{4\tau +4 +\epsilon } I_{ \{ \vert \eta_{12k} \vert \le n^{s}  \} } \right] <\infty \nonumber
\end{equation} 
and
\begin{equation}
\begin{aligned}
\text{Var} \left(  \eta_{12k}-\hat{\eta}_{12k} \right)&=\text{Var} \left(  \eta_{12k}I_{ \{ \vert \eta_{12k} \vert > n^{s}  \} } \right)\\
& \le E\left(  \eta_{12k}^{2}I_{ \{ \vert \eta_{12k} \vert > n^{s}  \} } \right)\\
& \le \frac{E\left( \vert \eta_{12k} \vert^{4\tau +4 +\epsilon}I_{ \{ \vert \eta_{12k} \vert > n^{s}  \} } \right)}{n^{s(4\tau +2 +\epsilon )}} \\
&\le C n^{-s(4\tau +2 +\epsilon )} \nonumber
\end{aligned}
\end{equation} 
as $n\to \infty$ due to $E \vert X_{1,1} \vert ^{8\tau+8+\epsilon } < \infty$  for some $ \epsilon  >0$. Then, we see from Fuk--Nagaev inequality that, for $\varepsilon>0$,
\begin{equation}
\begin{aligned}
&\sum_{n=1}^{\infty}P\left( \max_{1\le i< j\le p_{n}} \left \vert \sum_{k=1}^{n}\left(\eta_{ijk}- \hat{\eta}_{ijk}\right) \right \vert > \varepsilon \sqrt{n \log p_{n}} \right)\\
 \le &\sum_{n=1}^{\infty} p_{n}^{2} \cdot P\left( \left \vert \sum_{k=1}^{n}\left(\eta_{12k}- \hat{\eta}_{12k}\right) \right \vert > \varepsilon \sqrt{n \log p_{n}} \right)\\
\lesssim &\sum_{n=1}^{\infty}\frac{n p_{n}^{2} \cdot E\left \vert \eta_{12k}-\hat{\eta}_{12k} \right \vert ^{4\tau +4 +\epsilon }}{\left(\varepsilon \sqrt{n \log p_{n}} \right)^{4\tau +4 +\epsilon}}+\sum_{n=1}^{\infty} p_{n}^{2} \cdot \text{exp} \left( -C\frac{\varepsilon^{2} \log p_{n}}{\text{Var} \left(  \eta_{12k}-\hat{\eta}_{12k} \right)} \right)\\
\lesssim &\sum_{n=1}^{\infty} \frac{1}{n^{1+\epsilon/2}}+\sum_{n=1}^{\infty} \text{exp} \left(2\tau C\log n -Cn^{s(4\tau +2 +\epsilon )} \log p_{n} \right) < \infty \nonumber
\end{aligned}
\end{equation}
as $n \to \infty$ under the assumption that $0 < c_{1} \le p/ n^{\tau} \le c_{2} <\infty $  for any $ \tau  >0$. By the Borel--Cantelli lemma, we get
\begin{equation}
\lim_{n\to\infty}\frac{  \max_{1\le i< j\le p_{n}} \left \vert \sum_{k=1}^{n}\left(\eta_{ijk}- \hat{\eta}_{ijk}\right) \right \vert}{\sqrt{ n \log_{}{p_{n}}}}\le \varepsilon \quad  \text{a.s.} \nonumber
\end{equation}
By choosing $\varepsilon>0$ small enough, the proof of (\ref{a.s.three}) is completed.
\end{proof}

\subsection{Proof of Theorem \ref{thm2.4}}
\begin{proof}
We continue to use the notations in the proof of Theorem \ref{thm2.3}. Note that 
\begin{equation}
\begin{aligned}
Ee^{t_{0}\vert \eta_{ijk} \vert ^{\alpha}}\le &~E~\text{exp} \left( Ct_{0} \left \vert X_{i,k}^{2}+X_{j,k}^{2} +\sigma^{2} \right \vert ^{\alpha} \right)\\
\le & ~E~\text{exp} \left( Ct_{0}  \vert X_{i,k} \vert ^{2\alpha}+Ct_{0}  \vert X_{j,k} \vert ^{2\alpha} +Ct_{0} \sigma^{2\alpha}   \right)\\
=&~ C\cdot E~\text{exp} \left(2 Ct_{0}  \vert X_{i,k} \vert ^{2\alpha} \right) <\infty \nonumber
\end{aligned}
\end{equation}
due to $Ee^{t_{0}\vert X_{1,1} \vert ^{2\alpha}}<\infty$ for some $0<\alpha \le1/2$ and $ t_{0}>0$. Then, by Lemma \ref{L2}, we have
\begin{equation}\label{24}
\begin{aligned}
&\sum_{n=1}^{\infty} P\left( \max_{1\le i< j \le p_{n}} \sum_{k=1}^{n} \eta_{ijk} >\sqrt{(4+\varepsilon )n \log_{}{p_{n}}} \right)\\
 \le& \sum_{n=1}^{\infty} p_{n}^{2}\cdot P\left( \frac{1}{\sqrt{n}} \sum_{k=1}^{n} \eta_{12k} >\sqrt{(4+\varepsilon ) \log_{}{p_{n}}} \right)\\
\sim & \sum_{n=1}^{\infty} \frac{p_{n}^{2}}{ \sqrt{(4+\varepsilon ) \log_{}{p_{n}}}} e^{-\frac{(4+\varepsilon ) \log_{}{p_{n}}}{2}}\\
\le & \sum_{n=1}^{\infty} \frac{1}{p_{n}^{\varepsilon/2}}
= \sum_{n=1}^{\infty} o\left( \frac{1}{e^{\frac{\varepsilon}{2} n^{\frac{\alpha}{2-\alpha}} }} \right)< \infty
\end{aligned}
\end{equation}
as $n \to \infty$ since $\log_{}{p}=o( n^{\frac{\alpha}{2-\alpha}} )$. Next, by the Borel--Cantelli lemma, one has
\begin{equation}\label{a.s.sup}
\limsup _{n\to\infty}\frac{  \max_{1\le i< j\le p_{n}} \sum_{k=1}^{n}\eta_{ijk}}{\sqrt{ n\log_{}{p_{n}}}} \le \sqrt{4+\varepsilon}  \quad  \text{a.s.}
\end{equation}

Review the proof of (\ref{a.s.two}), the definitions of $c_{n2}$, $I$ and $B_{\alpha}$. For each $\alpha=(i,j) \in I$, set $Z_{\alpha}=\frac{1}{\sqrt{n}} \sum_{k=1}^{n} \eta_{ijk}$.  By Lemma \ref{L1}, we have
\begin{equation}\label{chenstein}
 P\left ( \max_{\alpha \in I}Z_{\alpha } \le  c_{n2}   \right )\le e^{-\lambda_{2}}  + u_{1}+u_{2}, \nonumber
\end{equation}
where
\begin{equation}
\begin{aligned}
\lambda_{2} &=\frac{ p_{n}\left ( p_{n}-1 \right )  }{2} P\left ( \frac{1}{\sqrt{n}} \sum_{k=1}^{n} \eta_{12k}>  c_{n2}   \right ),\\
u_{1} &\le p_{n}^{3}\cdot P\left ( \frac{1}{\sqrt{n}} \sum_{k=1}^{n} \eta_{12k} >  c_{n2}   \right )^{2},\\
u_{2}  & \le p_{n}^{3} \cdot P\left ( \frac{1}{\sqrt{n}} \sum_{k=1}^{n} \eta_{12k}>  c_{n2} ,\frac{1}{\sqrt{n}} \sum_{k=1}^{n} \eta_{13k} > c_{n2}  \right ).  \nonumber 
\end{aligned}
\end{equation}
By Lemma \ref{L2} and the same argument as in (\ref{24}), we have
\begin{equation}
\lambda_{2} \sim \frac{p_{n}\left ( p_{n}-1 \right )}{ 2\sqrt{2\pi}c_{n2}} e^{-\frac{c_{n2}^{2}}{2}} = o \left( p_{n}^{\frac{\varepsilon}{2}} \right)= o\left( e^{\frac{\varepsilon}{2} n^{\frac{\alpha}{2-\alpha}} } \right) \nonumber
\end{equation}
as $n \to \infty$. Furthermore, one can get that
\begin{equation}
u_{1} \le \frac{p_{n}^{3}}{2\pi c_{n2}^{2}} e^{-c_{n2}^{2}} \le \frac{1}{p_{n}^{1-\varepsilon}} = o\left( \frac{1}{e^{(1-\varepsilon) n^{\frac{\alpha}{2-\alpha}} }} \right) \nonumber
\end{equation}
for sufficiently large $n$. Write
\begin{equation}
\eta_{12k}+\eta_{13k}=\frac{\vert X_{1,1}-X_{2,1}\vert^{2}+\vert X_{1,1}-X_{3,1}\vert^{2}-4E\vert X_{1,1}\vert ^{2}}{\sqrt{2(E\vert X_{1,1}\vert^{4}+(E\vert X_{1,1}\vert^{2})^{2})}}. \nonumber
\end{equation}
By some calculations, we see
\begin{equation}
\begin{aligned}
E\left( \eta_{12k}+\eta_{13k} \right)&=0,\\ 
\text{Var} \left( \eta_{12k}+\eta_{13k} \right)&=2+2\text{Corr} (\vert X_{1,1}-X_{2,1}\vert^{2}, \vert X_{1,1}-X_{3,1}\vert^{2}):=\tilde{\sigma}^{2} \nonumber
\end{aligned}
\end{equation}
for each $k$. Moreover, we know $Ee^{t_{0}\vert \eta_{12k}+\eta_{13k} \vert ^{\alpha}} <\infty$ for some $0<\alpha \le 1/2$ and $t_{0}>0$. By Lemma \ref{L2}, we then have that
\begin{equation}
\begin{aligned}
u_{2}\le &~ p_{n}^{3} \cdot P\left( \frac{1}{\sqrt{n}} \sum_{k=1}^{n} \frac{\eta_{12k}+\eta_{13k}}{\tilde{\sigma}}>\frac{2c_{n2}}{\tilde{\sigma}} \right)\\
\sim &~ \frac{\tilde{\sigma} \cdot p_{n}^{3}}{2\sqrt{2\pi}c_{n2}}e^{-\frac{2(4-\varepsilon ) \log_{}{p_{n}}}{\tilde{\sigma}^{2}} } \\
\le &~ p_{n}^{3-\frac{2(4-\varepsilon )}{\tilde{\sigma}^{2}}} \nonumber
\end{aligned}
\end{equation}
as $ n \to \infty$. If $\text{Corr}(\vert X_{1,1}-X_{2,1}\vert^{2}, \vert X_{1,1}-X_{3,1}\vert^{2})<1/3$, then $3-\frac{2(4-\varepsilon )}{\tilde{\sigma}^{2}}<0$ and
\begin{equation}
\sum_{n=1}^{\infty} P\left ( \max_{\alpha \in I}Z_{\alpha } \le  c_{n2}  \right )\le \sum_{n=1}^{\infty} \frac{1}{p_{n}^{\varepsilon^{\prime\prime} }} = \sum_{n=1}^{\infty} o\left( \frac{1}{e^{\varepsilon^{\prime\prime}  n^{\frac{\alpha}{2-\alpha}} }} \right) < \infty \nonumber
\end{equation}
as $ n \to \infty$ for some constant $\varepsilon^{\prime\prime} >0$ depending on $\varepsilon$ and $\text{Corr}(\vert X_{1,1}-X_{2,1}\vert^{2}, \vert X_{1,1}-X_{3,1}\vert^{2})$. Therefore, we see from the Borel--Cantelli lemma that 
\begin{equation}
\liminf _{n\to\infty}\frac{  \max_{1\le i< j\le p_{n}} \sum_{k=1}^{n}\eta_{ijk}}{\sqrt{ n\log_{}{p_{n}}}} > \sqrt{4-\varepsilon}  \quad  \text{a.s.} \nonumber
\end{equation}
Combining the above inequality, (\ref{etam}) and (\ref{a.s.sup}), the desired conclusion follows from the arbitrariness of $\varepsilon$.
\end{proof}

\section{Applications}\label{sec5}
\subsection{$l^{q}$-norm}
Instead of studying the Euclidean norm distance, we will consider a more general distance, that is, $l^{q}$-norm distance in $\mathbb{R}^{n}$. For a $n$-dimensional random vector $\boldsymbol{X}$, define the $l^{q}$-norm of $\boldsymbol{X}$ by
\begin{equation}
\left \| \boldsymbol{X} \right \| _{q}=\left( \sum_{k=1}^{n} \vert X_{k} \vert^{q} \right)^{1/q}, \nonumber
\end{equation}
where $q\ge 1$. And the maximum interpoint $l^{q}$-norm distance is denoted by
\begin{equation}
M_{n,q}=\max_{1\le i< j\le p} \left \| \boldsymbol{X}_{i}- \boldsymbol{X}_{j} \right \|_{q}=\max_{1\le i< j\le p} \left( \sum_{k=1}^{n}\vert X_{i,k}-X_{j,k} \vert^{q}\right)^{1/q}. \nonumber 
\end{equation}
Similarly to Theorem \ref{thm2.3}, one can deduce the law of the logarithm for $M_{n,q}$ as follows.

\begin{theorem}\label{thm4.1}
Assume
\begin{flalign}
\quad &(i) \  0 < c_{1} \le p/ n^{\tau} \le c_{2} <\infty \ \text{for} \text{ any} \ \tau  >0; \nonumber &\\ 
\quad &(ii) \ E \vert X_{1,1} \vert ^{q(4\tau+4+\epsilon) } < \infty  \ \text{for} \ \text{some} \ \epsilon  >0; \nonumber &\\
\quad &(iii) \ \mathrm{Corr} (\vert X_{1,1}-X_{2,1}\vert^{q}, \vert X_{1,1}-X_{3,1}\vert^{q})<1/3; \nonumber &
\end{flalign}
where $c_{1}$ and $c_{2}$ are constants not depending on $n$. Then, the following holds as $n \to \infty$:
\begin{equation}\label{lq}
 \frac{  M_{n,q}^{q}-n E\left(\vert X_{1,1}-X_{2,1} \vert^{q}  \right)}{\sqrt{\mathrm{Var} (\vert X_{1,1}-X_{2,1}\vert^{q})n \log_{}{p}}} \to 2  \quad  \text{a.s.} 
\end{equation}
\end{theorem}
\begin{proof}
Reviewing the proof of Theorem \ref{thm2.3}, Theorem \ref{thm4.1} follows from changing $\eta_{ijk}$ to 
\begin{equation}
\eta_{ijk}=\frac{\vert X_{i,k}-X_{j,k} \vert^{q} -E\left(\vert X_{1,1}-X_{2,1} \vert^{q}  \right)}{\sqrt{\mathrm{Var} (\vert X_{1,1}-X_{2,1}\vert^{q})}}. \nonumber
\end{equation}
It is obvious that 
\begin{equation}
\max_{1\le i< j\le p}  \sum_{k=1}^{n}\eta_{ijk}= \frac{  M_{n,q}^{q}-n E\left(\vert X_{1,1}-X_{2,1} \vert^{q}  \right)}{\sqrt{\mathrm{Var} (\vert X_{1,1}-X_{2,1}\vert^{q})}}.  \nonumber
\end{equation}
Then, the proof of Theorem \ref{thm4.1} is completed. 
\end{proof}

By Theorem \ref{thm2.4} and the same method as above, we obtain the following result.

\begin{theorem}\label{thm4.2}
Assume
\begin{flalign}
\quad &(i) \  Ee^{t_{0}\vert X_{1,1} \vert ^{q\alpha}}<\infty \ \text{for some} \ 0<\alpha \le1/2, \ t_{0}>0 ; \nonumber &\\ 
\quad &(ii) \ \log_{}{p}=o\left( n^{\frac{\alpha}{2-\alpha}} \right); \nonumber &\\
\quad &(iii) \ \mathrm{Corr} (\vert X_{1,1}-X_{2,1}\vert^{2}, \vert X_{1,1}-X_{3,1}\vert^{2})<1/3. \nonumber &
\end{flalign}
Then, (\ref{lq}) holds as $ \to \infty$.
\end{theorem}

\subsection{Simulation resluts}
In this section, we carry out a Monte Carlo simulation on the law of logarithm for the maximum interpoint distance $M_{n}$ discussed in the previous section, to illustrate the validity of our results. We assume that $\{ X_{i,k}; 1\le i \le p, 1\le k \le n \}$ are i.i.d.~$N(0,1)$-distributed random variables. For each Monte Carlo iteration, we calculate the value of 
\begin{equation}
z=\frac{  M_{n}^{2}-2 n}{2\sqrt{2 n \log_{}{p}}}. \nonumber
\end{equation}

In our simulation, we perform $K=300$ Monte Carlo iterations and work on four different combinations of $(p,n)$ with $(p,n)$=(150,~100), $(p,n)$=(200,~200), $(p,n)$=(500,~250), and $(p,n)$=(600,~400). Figures 1 and 2 present scatter diagrams that compare the values of $z$ and 2. The observed results indicate that the value of $z$ are uniformly distributed around 2, with no discernible bias between the scatter plots and 2, a trend accentuated by the substantial values of both $p$ and $n$. This outcome serves as a compelling illustration of the validity of the results derived from the law of logarithm for the maximum interpoint distance $M_{n}$. The systematic exploration of various $(p,n)$ combinations enriches the understanding of the behavior of $z$ in the context of this simulation, adding depth and credibility to the findings.

\begin{figure}[htbp]
\centering
	\subfigure{
	\includegraphics[width=7cm,height=6cm]{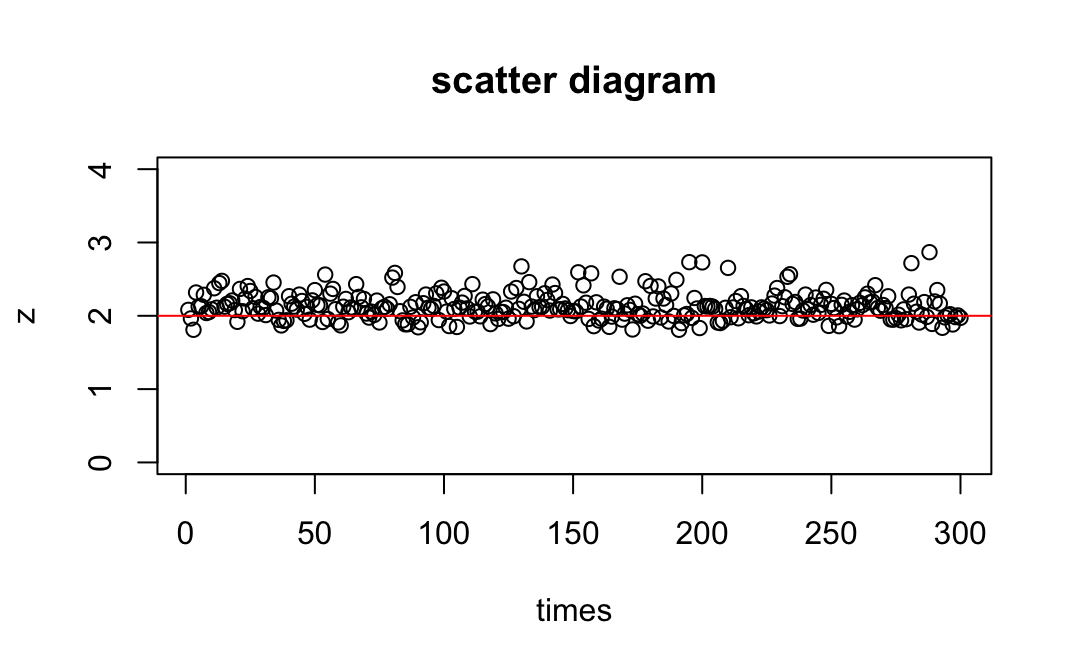} \label{Fig.6(b)}
}
	\subfigure{
	\includegraphics[width=7cm,height=6cm]{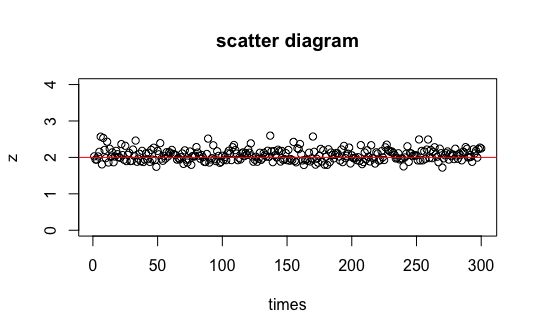} \label{Fig.6(b)}
}			
	\caption{Scatter diagrams corresponding to $(n,p)=(150,100)$ for the left picture and $(n,p)=(200,200)$ for the right. The straight line is $z=2$.}
\end{figure}
\begin{figure}[htbp]
\centering
	\subfigure{
	\includegraphics[width=7cm,height=6cm]{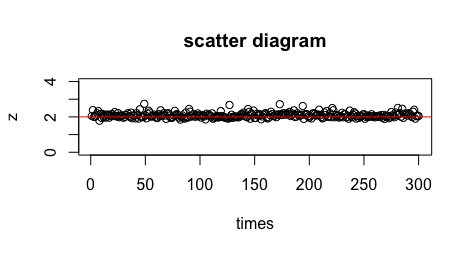} \label{Fig.6(b)}
}
	\subfigure{
	\includegraphics[width=7cm,height=6cm]{figure_3.png} \label{Fig.6(b)}
}			
	\caption{Scatter diagrams corresponding to $(n,p)=(500,250)$ for the left picture and $(n,p)=(600,400)$ for the right. The straight line is $z=2$.}
\end{figure}

\section*{Acknowledgements}
The authors thank anonymous reviewers for giving valuable comments and suggestions in improving the manuscript.
This paper was supported by National Natural Science Foundation of China (Grant No. 11771178, 12171198); the Science and Technology Development Program
of Jilin Province (Grant No. 20210101467JC); Technology Program of Jilin Educational Department during the "14th Five-Year" Plan Period (Grant No. JJKH20241239KJ) and Fundamental Research Funds
for the Central Universities.

\end{document}